\documentclass[EJP]{ejpecp} % replace ECP by EJP if needed.
\SHORTTITLE{Three-velocity coalescing ballistic annihilation}

\TITLE{Three-velocity coalescing ballistic annihilation} % \thanks is optional. Insert line breaks with \\

\AUTHORS{%
 Luis Benitez\footnote{University of Notre Dame.
    \EMAIL{lbenite3@nd.edu}}
  \and 
  Matthew Junge\footnote{Baruch College. \BEMAIL{Matthew.Junge@baruch.cuny.edu} }
  \and 
  Hanbaek Lyu\footnote{University of Wisconsin. \BEMAIL{hlyu@math.ucla.edu} }
  \and 
  Maximus Redman\footnote{Bard College. \BEMAIL{mr3274@bard.edu} }
  \and 
  Lily Reeves\footnote{Cornell University. \BEMAIL{zw477@cornell.edu} }
  }%AUTHORS
 % Separate items with ;

\AMSSUBJ{60K35} % Edit. Separate items with ;
\SUBMITTED{September 6, 2022} % Edit.
\ACCEPTED{April 17, 2023} % Edit.

\VOLUME{28}
\YEAR{2023}
\PAPERNUM{56}
\DOI{10.1214/23-EJP948}

\ABSTRACT{Three-velocity ballistic annihilation is an interacting system in which stationary, left-, and right-moving particles are placed at random throughout the real line and mutually annihilate upon colliding. We introduce a coalescing variant in which collisions may generate new particles. For a symmetric three-parameter family of such systems, we compute the survival probability of stationary particles at a given initial density. This allows us to describe a phase-transition for stationary particle survival.}

\usepackage{enumitem}
\usepackage{mathtools}
\usepackage{comment}
\mathtoolsset{showonlyrefs}

\newcommand{\Z}{\mathbb{Z}}

\newcommand{\E}{\mathbb{E}}
\renewcommand{\P}{\mathbb{P}}

\newcommand{\f}{\frac}
\newcommand{\lrl}{\longleftrightarrow}
\newcommand{\lrlf}{\overset{1}{\longleftrightarrow}}
\newcommand{\lrlb}{\overset{\h}{\longleftrightarrow}}
\newcommand{\ra}{\rightarrow}

\newcommand{\la}{\leftarrow}

\newcommand{\Laf}{\overset{1}{\Leftarrow}}

\newcommand{\ind}[1]{\mathbf 1 \{#1 \}}
\renewcommand{\emptyset}{\varnothing}
\renewcommand{\phi}{\varphi}

\usepackage{hyperref}
\hypersetup{colorlinks}
\newcommand{\HL}[1]{{\color{red}{#1}}}

\newcommand{\cev}[1]{\reflectbox{\ensuremath{\vec{\reflectbox{\ensuremath{#1}}}}}}
\renewcommand{\b}{\bullet}
\renewcommand{\L}{\cev{\b}}
\newcommand{\R}{\vec{\b}}
\newcommand{\B}{\dot{\b}}
\newcommand{\h}{\hat{\b}}
\newcommand{\BB}{\mathcal B}
\renewcommand{\AA}{\mathcal A}
\newcommand{\p}{\mathfrak p}

\usepackage{theoremref}
\begin{document}

%%%%%%%%%%%%%%%%%%%%%%%%%%%%%%%%%%%%%%%%%%%%%%%%%%%%%%%%%%%%%%%%%%%
%%                                                               %%
%% No need for \maketitle.                                       %%
%%                                                               %%
%%%%%%%%%%%%%%%%%%%%%%%%%%%%%%%%%%%%%%%%%%%%%%%%%%%%%%%%%%%%%%%%%%%

%%%%%%%%%%%%%%%%%%%%%%%%%%%%%%%%%%%%%%%%%%%%%%%%%%%%%%%%%%%%%%%%%%%
%%                                                               %%
%% Please replace what follows by the body of your article       %%
%% (up to the bibliography):                                     %%
%%                                                               %%
%%%%%%%%%%%%%%%%%%%%%%%%%%%%%%%%%%%%%%%%%%%%%%%%%%%%%%%%%%%%%%%%%%%

\section{Introduction}

In \emph{ballistic annihilation}, particles are placed throughout the real line with independent and identically distributed spacings sampled from a continuous distribution. Each particle is assigned a velocity at which it moves from the onset. When particles collide, they mutually annihilate and are removed from the system. While a variety of velocity distributions have been studied \cite{anti-particle,b5,  b1, b9}, a standard way to assign velocities is independently from $\{-1,0,1\}$ where velocity 0 is assigned with probability $p\in [0,1)$, and velocities $\pm1$ symmetrically with probability $(1-p) /2$. 
%Note that the collision pairings are invariant under translation, so these velocities are representative of what occurs in systms with velocities $v - v', v, v+v'$ for any $v,v' \in \mathbb R$. 
We will refer to this system as \emph{symmetric three-velocity ballistic annihilation} (BA).  Particles with velocity 0 will be referred to as \emph{blockades} and those with velocities $\pm1$ as \emph{arrows}. When necessary, we further specify the direction of an arrow as \emph{right} ($+1$) and \emph{left} ($-1$). In this present work, we extend BA dynamics to include collisions that sometimes generate new particles (see \eqref{eq:rules}).

Many intriguing features of BA were inferred by physicists in the 1990s \cite{b4, b8}. Several decades later the papers \cite{ST, BGJ, bullets, bullets2} brought renewed attention to the problem. Recently, a mathematical approach was developed by Haslegrave, Sidoravicius, and Tournier to rigorously justify these inferences \cite{HST}. Several basic quantities in ballistic annihilation continue to evade analysis. A major difficulty is that the order in which collisions occur is sensitive to perturbations; changing the velocity of a single particle can have a cascading effect. This makes it difficult to couple processes with different parameters and to prove continuity of basic statistics.  For example, the ideas from \cite{HST} have so far only been partially extended to asymmetric three-velocity ballistic annihilation in which left and right arrows occur with different probabilities \cite{JL, HT}. %Moreover, the extension to this setting required new insights.  
%We seek to better understand the reach, as well as the limits, of the approach from \cite{HST}. 
Seeking to better understand the reach, as well as limits, of the approach in \cite{HST}, we generalize their result concerning the location of the phase transition in BA to coalescing systems. Followup work concerning universality of the phase transition in coalescing ballistic annihilation can be found in \cite{cruzado2021arrivals}.
%in which collisions sometimes generate new particles.
%We will discuss in more detail below, but ballistic annihilation is very sensitive to perturbations; changing the velocity of a single, let alone infinitely many, particle can have a cascading effect. 
%For example, little is known about the three-velocity system with asymmetric probabilities of left and right arrows \cite{JL}. 

\subsection{Notation and key quantities} 
%A convenient notation for ballistic annihilation was introduced in \cite{HST}. 
For each nonnegative integer $k$ we let $\b_k$ represent the $|k|$th particle to the right or left of the origin ($k>0$ for right and $k<0$ for left) whose initial location is denoted by $x_k \in \mathbb R$. We set $x_0 = 0$ and sample $x_k$ so that the $(x_k - x_{k-1})$ are independent according to a continuous distribution with support contained in $(0,\infty)$. Then for any integer $m\in \mathbb{Z}$, $(x_{i+m})_{i\in \mathbb{Z}}$ has the same law as $(x_{i})_{i\in \mathbb{Z}}$. Particle types are then assigned independently at each $x_k$. The process on the whole line $\mathbb{R}$ is invariant under translating the particle indices. We will frequently refer to the events
\begin{align*}
    \B_k &= \{\b_k \text{ is a blockade}\},   
    \R_k = \{\b_k \text{ is a right arrow}\}, \text{ and }
    \L_k = \{\b_k \text{ is a left arrow}\}.
\end{align*}
Collision events and visits to a location $ u \in \mathbb R$ are specified by
\begin{align}
    \{\b_j \lrl \b_k\} &= \{\b_j \text{ and } \b_k \text{ mutually annihilate}\} \\
    \{ u \la  \b\} &= \{ u \text{ is visited by a particle  from the right}\}.
\end{align}
{Note that the collision event $\{\b_j \lrl \b_k\}$ depends only on finitely many initial particles in the interval containing $[x_{j}, x_{k} ]$ as its middle third, so it is well-defined. Consequently, the events $\{\b_{j}\leftrightarrow \b \}:=\cup_{k>j} \{\b_j \lrl \b_k\}$ are also well-defined.} Note that we count an arrow destroying a blockade as also visiting the site housing the blockade, so $\{\B_k \la \b\}\subseteq {\{x_k \la \b\}}$. 
It is sometimes advantageous to restrict to the system with only the particles started in a specified interval $I \subseteq \mathbb R$. We notate this restriction by including $I$ as a subscript on the event, for example, $(\B_j \lrl \L_k)_{[x_j, x_k]}$ is the event that $\b_j$ is a blockade which mutually annihilates with a left arrow started at $x_k$ when restricted to only the particles in $[x_j,x_k]$. 

A fundamental statistic associated to BA is the probability $\theta(p)$ that the origin is not visited conditional on $\B_0$.  We define this quantity formally as
\begin{align}
\theta(p):=(1-q)^{2},\qquad   q=q(p) := \P_p( (0 \leftarrow \b)_{(0, \infty)}). 
\label{eq:theta}
\end{align}
The Birkhoff Ergodic Theorem ensures that the limiting density of surviving blockades is $p \theta(p)$.
Unless stated otherwise, all of the events  we consider hereafter are one-sided on $(0,\infty)$. Accordingly, we drop the subscript $(0,\infty)$ from our event notation. 

Physicists inferred that $\theta$ undergoes a phase transition as the initial density of blockades is varied \cite{b8}. Formally, we define the critical values
\begin{align}
p_c^- = \inf\{ p \colon \theta(p) >0\} \qquad \text{ and } \qquad  p_c^+ = \sup \{ p \colon \theta(p) =0 \}.
\end{align}
Despite being intuitively plausible, there is no known coupling that proves $\theta$ is increasing in $p$. Thus, $p_c^-$ and $p_c^+$ may not coincide. However, when they do, we denote the location of the phase transition by $p_c$. 
It was deduced in \cite{b4} and further supported by calculations in \cite{b8} that for BA it holds that $p_c = 1/4$. The breakthrough from Haslegrave, Sidoravicius, and Tournier provided, among many remarkable results, a rigorous probabilistic proof that $p_c =1/4$ \cite{HST}. 

\subsection{Coalescing ballistic annihilation}
We consider ballistic motion in which collisions sometimes result in the generation of new particles. This is inspired by earlier work from physicists \cite{b3, b7, blythe2000stochastic}. However, none of these works considered the three-velocity setting. We remark that \cite{HST} allowed for a coalescence rule in which a particle is selected uniformly at random to survive a triple collision. The primary reason for considering this case was to resolve technical difficulties that arise in the presence of triple collisions, rather then investigate coalescence dynamics. Note that when the spacings between particles are sampled from an atomless distribution, there are almost surely no triple collisions.

%To describe such systems, which we refer to as \emph{three-velocity coalescing ballistic annihilation}, 
We require more notation to describe coalescing systems. The initial conditions with particles $\b_k$ at $x_k$ for $k \in \mathbb Z$ assigned velocities from $\{-1,0,1\}$ remain unchanged. We denote two particles meeting at the same location by $\b_j - \b_k$. Upon meeting, a reaction takes place. Either the particles coalesce and form a new particle with an independently sampled new velocity, or the particles mutually annihilate. 

In general, there are three types of collisions: $\R - \B$, $\B - \L$, and $\R - \L$. 
Each collision may result in one of four reactions: generating a left arrow, right arrow, blockade, or mutual annihilation (denoted by $\emptyset$). 
%We refer to the family of all such systems as \emph{general coalescing ballistic annihilation.} Such systems would require 12 parameters to describe the probabilities of the different collisions and reactions. 
%We were unable to analyze many of these systems. For example, any system lacking reflection symmetry is currently out of reach for reasons similar to those expressed in \cite{JL}. Additionally, we were unable to satisfactorily analyze $\B - \R$ and $\B - \L$ reactions in which the blockade survives (e.g., $\B - \R \implies \B$) nor were we able to obtain results for systems in which a particle moving in the opposite direction of the reactant is generated (e.g., $\B - \R \implies \L$). Such cases do not appear to result in the same degree of self-similarity as the other cases. We feel that there is some hope for analyzing reactions in which blockades survive. The major difficulty with that case is discussed more in \thref{rem:block}. 

We define the \emph{three-parameter coalescing ballistic annihilation} (TCBA) covered by our results.
Fix parameters $0\leq a,b,\alpha <1$ with $a+b\leq 1$. Using the notation 
\begin{align*}
    \b - \b \implies \Theta, \quad x
\end{align*}
to denote a collision resulting in an outcome $\Theta \in \{ \B, \R, \L, \emptyset \}$ with probability $x$, we have the following collision rules:

\begin{minipage}{.5\linewidth}
\begin{equation}
    \R - \L  \implies 
    \begin{cases}
        \L, &  a/2\\
        \R, &  a/2 \\
        \B, &  b \\
        \emptyset, & 1- (a+b)
    \end{cases}
\end{equation}
\end{minipage}%
\begin{minipage}{.5\linewidth}
\begin{align}
    \B - \L &\implies  
    \begin{cases} 
        \L, &  \alpha\\
        %\R, & \text{with probability } 0 \\
        %\B, & \text{with probability } y \\
        \emptyset, & 1-\alpha
    \end{cases}\\
    \R - \B &\implies  
        \begin{cases} 
            \R, &  \alpha\\
            %\R, & \text{with probability } 0 \\
            %\B, & \text{with probability } y \\
            \emptyset,& 1-\alpha
    \end{cases}. \label{eq:rules}
\end{align}
\end{minipage}
So TCBA allows for arrows to survive collisions with blockades and other arrows or to generate a blockade after colliding with an arrow. Note that BA is the special case $a=b=\alpha=0$. 

For all but the $\R - \L \implies \B$ reaction, it is mathematically equivalent to view coalescence as one of the particles surviving the collision. Taking this perspective, we have the head of the arrow point to the particle that is destroyed. For example, $\R_j \to \L_k$ denotes the event that the left arrow started at $x_k$ is destroyed by the right arrow started at $x_j$. The survived particle is still denoted by $\R_j$. If mutual annihilation occurs, then we continue to write $\b \lrl \b$. We denote the case in which two arrows collide and generate a blockade by $\R_m \lrlb \L_n$. We denote blockades generated from such collisions by $\h_{m,n}$, and denote a generic blockade generated from such a reaction by $\h$. 
% \begin{align}
%     \R - \L& \implies \begin{cases} 
%                                     \L, &  a/2\\
%                                     \R, &  a/2 \\
%                                     \B, &  b \\
%                                     \emptyset ,& 1- (a+b)
%                                 \end{cases} 
% \end{align} 
% \begin{align}
%     \B - \L &\implies  \begin{cases} 
%                                     \L, &  x\\
%                                     %\R, & \text{with probability } 0 \\
%                                     %\B, & \text{with probability } y \\
%                                     \emptyset ,& 1-x
%                                 \end{cases}
%                                 \text{\qquad  and \qquad }
%     \R - \B \implies  \begin{cases} 
%                                     \R, &  x\\
%                                     %\R, & \text{with probability } 0 \\
%                                     %\B, & \text{with probability } y \\
%                                     \emptyset,& 1-x
%                                 \end{cases}. \label{eq:rules}
% \end{align}

%\subsection{Main result}

% \begin{figure}
%     \centering
%     \includegraphics[width = 8cm]{g.pdf}
%     \caption{The solutions to $g(u,v(u))=0$ from \eqref{eq:g} when $a=b=x=y=1/4$. The line $y=1$ and point $p_* = 9/38$ are also included. Graphically it is clear that $g$ is consistent. However, for this specific example, let alone the general case, consistency is difficult to confirm.}
%     \label{fig:g}
% \end{figure}

Recall that $(0 \la \b)$ is implicitly restricted to the positive real line. For TCBA, we define $q = q(a,b,\alpha,p) := \P(0 \la \b)$ and $\theta$ as in \eqref{eq:theta}:
\begin{align} 
\theta = \theta(a,b,\alpha,p) := (1-q)^2.\label{eq:theta_consideration}
\end{align}
We define $p_c = p_c(a,b,\alpha)$ as the value of $p_c^- = p_c^-(a,b,\alpha) :=\inf \{ p \colon q(a,b,\alpha,p)<1\}$ and $p_c^+ = p_c^+ (a,b,\alpha) := \sup \{ p \colon q(a,b,\alpha,p) = 1\}$ when they coincide. 
%Following \cite{HST}, we study bounds on these quantities.
%For certain choices of parameters we can verify that this criteria holds.
Our main result gives formulas for $q$ and $p_c$.

\begin{theorem} \thlabel{thm:sc}
For any TCBA it holds that 
\begin{align}
p_c = p_c(a,b,\alpha) &= \frac{1 - b (1-\alpha)}{4 - 3\alpha  - (a+b)(1-\alpha)}
\label{eq:pc}
\end{align}
with $q(p) = 1$ for $p \leq p_c$ and 
\begin{align}
q(p) = \frac{\sqrt{(1-\alpha) \left(b (1-p)^2- p (a (1-p)+p \alpha-1)\right)}-p (1-\alpha)}{(1-\alpha) ((1-a) p+b (1-p))}\label{eq:q}
\end{align}
for $p > p_c$.
\end{theorem}
%While not all that illuminating for a particular choice of parameters, the overall 
This generalizes \cite[Theorem 1]{HST} in which the formula $q(p) = p^{-1/2} -1$ for $p \geq 1/4$ and otherwise $q(p)=1$ is established. 
 Given the notorious sensitivity of BA to perturbation, it is noteworthy that we can describe systems whose local behavior is markedly different from BA.  The form of \eqref{eq:pc} illustrates how the location of the phase transition depends in a subtle way on the coalescence rules. This suggests that it would be difficult to infer $p_c$ from heuristic arguments such as those given in the introductions of \cite{b4, BGJ} for BA. Note that there are parameter choices that result in arbitrarily small and large values of $p_c$. See Figure \ref{fig:q} for a depiction of the function $q$ for various parameter choices. 
 
 Although \thref{thm:sc} extends the main result from \cite{HST} to ballistic systems with coalescence, it is still restricted to systems with reflection symmetry. For asymmetric three-velocity ballistic annihilation without coalescence---for example, left and right particles have different speeds or probabilities of occurring---universal bounds for $p_{c}$ were obtained in \cite{JL}. However, it seems to be difficult to establish a sharp phase transition and/or formulas for $p_{c}$ and $q(p)$ in asymmetric cases. The lack of reflection symmetry introduces a number of additional quantities to be solved explicitly. Another interesting extension is to allow blockades to survive collisions (e.g., $\B - \R \implies \B$) or particle moving in the opposite direction of the reactant is generated (e.g., $\B - \R \implies \L$). Such cases do not appear to offer the same sort of renewal as the cases we cover in \thref{thm:sc}. 
 %We feel that there is some hope for analyzing reactions in which blockades survive. 
 This is discussed more in \thref{rem:block}. After posting an earlier version of our article, a solution to the process in which blockades survive multiple collisions was found in \cite{junge2022non}.

\begin{figure}
    \centering
    \includegraphics[width = .6 \textwidth]{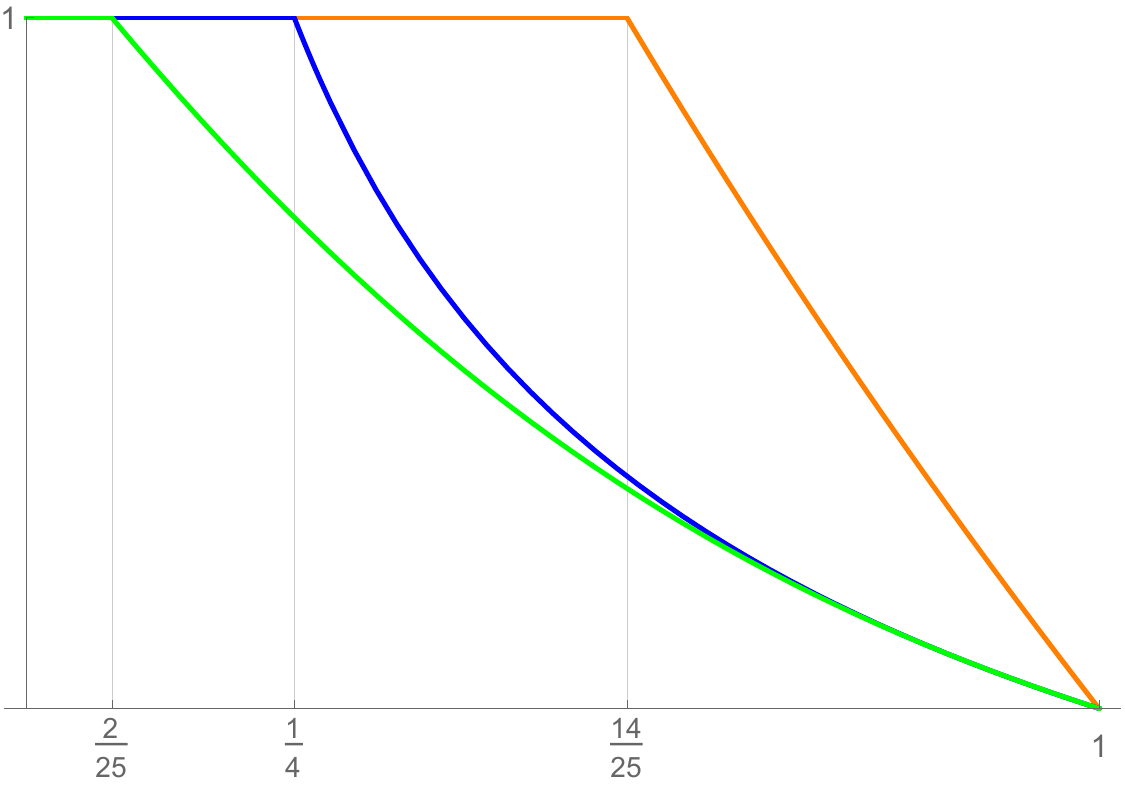}
    \caption{Plots of the formula for $q(p)$ from \thref{thm:sc} for different parameter choices $(a,b,\alpha)$. The horizontal axis is $p$, and the vertical axis is the probability. The green curve is for $(1/8,3/4,0)$ which gives $p_c = 2/25$. The blue curve is for $(0,0,0)$, so the usual BA, which has $p_c=1/4$. The orange curve is $(1/4,1/2,3/4)$ which gives $p_c = 14/25$.}
    \label{fig:q}
\end{figure}

\subsection{Proof methods}
The high-level idea for establishing the phase transition is that first we derive an equation involving only $q, a, b, \alpha$, and $p$. Then, we analyze the roots (in $p$) of this equation to find $p_c$ and show that there is indeed a phase transition.

Our first step is deriving the identity 
    \begin{align}
    0 = (1-q)g(p,q)\label{eq:id}
    \end{align} 
in \thref{prop:q} for an explicit function $g$. To obtain the identity we partition $q$ in terms of the velocity assigned to $\b_1$. Recall that all events below are restricted to the one-sided process on $(0,\infty)$:
    \begin{align}
    q &= \P((0 \leftarrow \b) \wedge \L_1) + \P((0 \leftarrow \b) \wedge \B_1) + \P((0 \leftarrow \b) \wedge \R_1) \label{eq:q_rec}.
    \end{align}
In \thref{lem:partition}, each of these three terms is expanded into a formula involving the parameters. These formulas are derived by further partitioning on the type of particle that destroys $\b_1$ and then observing some form of renewal. For example, the interdistances and types of particles started in $(x_j,\infty)$ are conditionally independent of the event $(\B_1 \lrl \L_j)$. Thus, the probability the origin is visited at least once is again equal to $q$.

While in the same spirit as the proof of \cite[Proposition 5]{HST}, where the identity $q = 1- p( 1+q -q^2 -q^3)$ is derived, the coalescing case is more complicated.  An important tool is a mass transport principle (\thref{prop:mtp}) derived from translation invariance. The use of the mass transport principle is inspired by what was done in \cite{JL}. Note that Haslegrave and Tournier \cite{HT} also found an application of the mass transport principle for computing a moment generating function associated to BA.   

The identity at \eqref{eq:id} implies that $q$ is either equal to $1$ or a root of $g$. Let $p_*$ be the claimed value of $p_c$ in \thref{thm:sc}. It is straightforward to prove that the roots of $g$ fall outside of $[0,1]$ for $p \in [0, p_*)$, and so $q=1$ on that interval. Moreover, at $p=p_*$ the only positive root of $g$ is equal to $1$. So we additionally have $q(p_*)=1$. Recall that there is no known direct proof that $q$ is continuous. On the interval $p\in(p_*,1)$, we cannot easily rule out the possibility that $q$ jumps between $1$ and a root of $g$. To rule out this pathology, we generalize an approach from \cite{HST} and prove that the set $I=\{ p\in (p_*,1) \colon q(p) <1\}$ is both open and closed in the subspace topology on $(p_*,1)$. These observations along with the fact that $I$ is nonempty (\thref{lem:pc+}) imply that $I = (p_*,1)$ and thus $q$ must be equal to the unique positive root of $g$ on this interval. 

The proof that $I^c$ is open follows what was done in \cite{HST}; we approximate $q$ from below with continuous functions $q_k =\P( (0 \la \b)_{(0,x_k]})$ to write $I^c$ as a union of sets which are open because of an explicit comparison to the nonnegative root of $g$. % $\cup_{k\geq 1} \{p \colon q_k(p) > \xi(p)\}$.
The proof that $I$ is open relies on a necessary and sufficient condition for blockades to survive with positive probability that is similar to what was done in \cite{HST}. Coalescence again introduces new challenges. \cite[Lemma 12]{HST} concerns showing that the net number of surviving blockades versus left arrows is superadditive when combining processes on adjacent intervals. This no longer holds with TCBA since arrows can destroy multiple blockades. We introduce a weighting scheme to account for this in \thref{lem:sa}. The argument concludes by showing that $q(p) <1$ if and only if there are on average more surviving blockades than arrows in our weighting scheme. This condition is different and a little more efficient than what was used for the analogue in \cite{HST}. See \thref{rem:sa}. We use this condition to prove that $I$ is open in \thref{lem:open} and that $I$ is nonempty in \thref{lem:pc+}. 

\subsection{Determining firstness} \label{sec:quiver}
Since arrows may survive multiple collisions, we give more detail concerning how reactions are decided. Arrow-arrow collisions are decided in some generic way that is compatible with \eqref{eq:rules} and consistent when combining the process restricted to different intervals. For example, left arrows carry a queue of instructions for what reaction occurs for each right arrow they meet. An upcoming lemma (\thref{lem:sa}) requires particular care with how arrow-blockade reactions are decided. To facilitate the presentation of the proof of \thref{lem:sa}, we associate to each arrow $\b_i$ a \emph{quiver} of $\sigma_i$ \emph{sharp arrows}. Each sharp arrow represents one blockade that $\b_i$ is able to destroy. Accordingly, $\sigma_i \sim \text{Geometric}_1(1-\alpha)$ is distributed as a geometric random variable supported on $1,2,\hdots$ with success parameter $1-\alpha$. 

Each arrow in the quiver has the same direction and position as $\b_i$. If $\b_i$ is destroyed by another arrow, then all arrows in the quiver are destroyed as well. Whenever $\b_i$ meets a blockade, one of the sharp arrows from its quiver mutually annihilates with the blockade. When the last sharp arrow in the quiver of $\b_i$ is destroyed, so is $\b_i$.  It is straightforward to verify that the quiver formulation for deciding collisions is equivalent to TCBA according to \eqref{eq:rules}.

Lastly, it will be necessary to distinguish whether the first arrow that visits a location would destroy and survive, or mutually annihilate with the next blockade it meets. Accordingly, for $u \in \mathbb R$, we introduce the events
\begin{align} 
    \{u \Laf \b \}& := \left\{ \begin{matrix} \text{$u$ is first visited by an arrow whose quiver} \\ \text{contains at least two sharp arrow} \end{matrix} \right\} \\
    \{u \lrlf \b \}& := \left\{ \begin{matrix} \text{$u$ is first visited by an arrow whose quiver} \\ \text{contains exactly one sharp arrow} \end{matrix} \right\}.
\end{align} 

\subsection{Organization}
%Section \ref{sec:quiver} introduces some additional notation for TCBA. 
In Section \ref{sec:recursion}, we prove a mass transport principle for TCBA. We then use this to derive a recursive formula for $q$ which leads to the identity at \eqref{eq:id} in \thref{prop:q}. Section \ref{sec:g} contains additional information about $g$ and its roots. Section \ref{sec:regularity} proves regularity of the set $I = \{ p \in (p_*,1) \colon q(p) <1\}$. Finally, Section \ref{sec:main} combines the other sections to give a short proof of \thref{thm:sc}.

\section{Recursion} \label{sec:recursion}
We begin by stating and proving the two main tools for obtaining a recursive expression for $q$.

\begin{proposition}[Mass Transport Principle] \thlabel{prop:mtp}
    Consider a family of non-negative random variables Z(m,n) for integers $m,n \in \Z$ such that its distribution is diagonally invariant under translation, i.e., for any integer $\ell$, $Z(m+\ell,n+\ell)$ has the same distribution as $Z(m,n)$. Then for each $m \in \Z:$
    \begin{align}
        \E \sum\limits_{n \in \Z} Z(m,n) = \E \sum\limits_{n \in \Z} Z(n,m) 
    \end{align}
\end{proposition}
\begin{proof}
Using Fubini's theorem and translation invariance of $\E [Z(m,n)]$, we obtain
\begin{align}
    \E \sum\limits_{n \in \Z} Z(m,n) =  \sum\limits_{n \in \Z} \E [Z(m,n)] = \sum\limits_{n \in \Z} \E [Z(2m-n,m)]
    = \sum\limits_{n \in \Z} \E [Z(n,m)] = \E \sum\limits_{n \in \Z} Z(n,m).
\end{align}
\end{proof}

\begin{proposition} \thlabel{prop:mc}
Let $c = 1- (a+b)$. The following equations hold so long as the parameters in the denominators are nonzero:
\begin{align}
    \f 1{a/2} \P( (\R_1 \la \L)_{(0, \infty)}) &= \f 1 b \P( (\R_1 \lrlb \L)_{(0, \infty)})= \f 1c \P((\R_1 \lrl \L)_{(0, \infty)}), \label{eq:rml}\\
    \P((u \la \b)_{(u, \infty)}) &= \f 1 \alpha \P((u \Laf \b)_{(u, \infty)}) = \f 1{1-\alpha} \P((u \lrlf \b)_{(u, \infty)}). \label{eq:smcf}
\end{align}
\end{proposition}
\begin{proof}
These are simple derivations from the probabilities of the various outcomes that occur when two particles meet. 
\end{proof}
We define a few auxiliary probabilities that will be useful when expanding \eqref{eq:q_rec}:
\begin{align}
% w &= P(\B_1 \la \b \mid \B_1) \\
s &:= P( (0 \la \b) \wedge (\R_1 \lrl \B_{k} \,\, {\textup{for some $k>1$}}) ) + \P( (0 \la \b) \wedge (\R_1 \lrl  \h))\\
r &:= P( (0 \not \la \b) \wedge (\R_1 \lrl \B_{k} \,\, {\textup{for some $k>1$}})) + \P( (0 \not \la \b) \wedge (\R_1 \lrl  \h)).
\end{align}

% {(\color{red}Here writing the first probability just as $\R_{1}\leftrightarrow \B$ is incorrect, as we never distinguished $\R_{1}\leftrightarrow \B$ from $\R_{1}\leftrightarrow \hat{\b}$. Need to be explicit that the former event only accounts for colliding with initial blockades.  ) }

\begin{lemma} \thlabel{lem:partition}
%Let $s$ and $r$ be as in \thref{lem:wsr}. 
Suppose that $c = 1- (a+b)>0$. For TCBA restricted to $(0,\infty)$ it holds that
\begin{align}
    \P((0 \la \b) \wedge \L_1 ) & = \f {1-p}2 \label{eq:pol}\\
    \P((0 \la \b) \wedge \B_1 )&= \alpha p q + (1-\alpha) p q^2 \label{eq:p0b}\\
    \P((0 \la \b) \wedge \R_1 )&=  \left( q +  \f {a/2} c    +  \f b c \f{\P( (0 \la \b) \wedge \B_1)}{p} \right)\P( \R_1 \lrl \L) + s    \label{eq:por} \\ 
    \P(\R_1 \lrl \L) &= \f{\f{1-p}2 - s - r}{1 + \f{a/2}c + \f bc}
    = \frac{c \left(p q^2 (1-\alpha)-2 p q (1-\alpha)-p+1\right)}{-a+b (2-q) q (1-\alpha)+2}\label{eq:prml}
\end{align}
\end{lemma}

\begin{proof}[Proof of \eqref{eq:pol}]
This is simply the observation that $\P( ( 0 \la \b) \wedge \L_1) = \P(\L_1).$
\end{proof}

\begin{proof}[Proof of \eqref{eq:p0b}]
We condition on $\{\B_1\}$ and consider the two ways that $\B_1$ can be annihilated: 
\begin{align}
    \P( (0 \la \b) \wedge \B_1 ) &=   p \P(\B_1 \la \L \mid \B_1)+ p\P( (0 \la \b) \wedge (\B_1 \lrl \L) \mid \B_1).
\end{align}
By \eqref{eq:smcf}, $\P(\B_1 \la \L \mid \B_1) = \P(x_1 \Laf \L) = \alpha q$. For the second term, 
\begin{align*}
    \P((0\la \b) \wedge (\B_1 \lrl \L) \mid \B_1) = \sum_{k> 1}\P((x_1 \lrlf \L_k)_{(x_1, x_k]} \wedge (x_k \la \b)_{(x_k, \infty)}).
\end{align*}
Since the events on $(x_1, x_k]$ and $(x_k, \infty)$ are independent, we have
\begin{align*}
    \P((0\la \b) \wedge (\B_1 \lrl \L) \mid \B_1) 
    &= \sum_{k>1} \P((x_1 \lrlf \L_k)_{(x_1, x_k]}) q \\
    &= \P(x_1 \lrlf \L) q = (1-\alpha) q^2.
\end{align*}
The last equality uses \eqref{eq:smcf} again. Putting both terms together, we obtain the claimed formula.
\end{proof}

\begin{proof}[Proof of \eqref{eq:por}]
We partition on the various ways that $\R_1$ is destroyed to write $\P( ( 0 \la \b) \wedge \R_1)$ as
\begin{align}
 \P((0 \la \b) \wedge (\R_1 \lrl \L) ) & +  \P((0 \la \b) \wedge (\R_1 \la \L) ) + \P((0 \la \b) \wedge (\R_1 \lrlb \L) )  \label{eq:L}  \\
 &+ \P((0 \la \b) \wedge (\R_1 \lrl \B) )  + \P((0 \la \b) \wedge (\R_1 \lrl \h ) ). \label{eq:B}
\end{align}
 
The first term of \eqref{eq:L} is equal to $\P(\R_1 \lrl \L) q$ via the renewal that occurs after $\R_1 \lrl \L$. For the second term in \eqref{eq:L}, the left arrow that destroys $\R_1$ will reach $0$. The change of measure in \thref{prop:mc} gives $\P( \R_1 \la \L) = ((a/2)/c)\P(\R_1 \lrl \L)$. For the third term in \eqref{eq:L}, we note that \thref{prop:mc} ensures that $\P(\R_1 \lrlb \b) =  (b/ c)\P( \R_1 \lrl \L)$. Next, after a blockade is generated by $\R_1 \lrlb \b$, having $0$ visited is a translation of the event $\{(0 \la \b) \wedge \B_1\}$ conditional on $\B_1$ already being present (since a blockade was generated). Thus, we write the probability as $\P((0 \la \b) \wedge \B_1)/p$. 

Lastly, observe that the two terms at \eqref{eq:B} have sum equal to $s$.
\end{proof}

\begin{proof}[Proof of \eqref{eq:prml}]
%The proof uses the fact that any arrow in the system is destroyed with probability $1$. 
The proof uses the fact that right arrows in $(0,\infty)$ is destroyed with probability $1$. This is proven in \cite[Lemma 3.4]{ST} and the same observation holds for TCBA.
%This is proven in \cite{ST} or \cite{HST} and the same observation holds for TCBA. 
The basic reason right arrows cannot survive in the right half-line is that, if they did, symmetry ensures left arrows also survive in the left half-line. Ergodicity gives a positive density of left and right surviving arrows in the full-line, which is a contradiction. This lets us form the partition
\begin{align}
\P(\R_1) = \f {1-p}{2} &= \P( \R_1 \lrl \B) + \P( \R_1 \la \L) + \P(\R_1 \lrl \L) + \P( \R_1 \lrlb \L).
\end{align}
The first term on the right side is equal to $s+r$. The last three terms can be transformed using \thref{prop:mc} which gives
\begin{align}
\f{1-p}{2} &= s +r + \f{a/2}{c} \P( \R_1 \lrl \L) + \P(\R_1 \lrl \L) + \f{b}{c} \P( \R_1 \lrl \L). \label{eq:Rmc}
\end{align}
Solving for $\P( \R \lrl \L)$ gives the first claimed equality in \eqref{eq:prml}. 

To obtain the final expression that involves only parameters and $q$, we use formulas for $s$ and $r$ which we will derive in \thref{lem:wsr}. Recall that $a+b+c=1$ and notice that the formulas for $s$ and $r$ depend only on parameters, $q$, and $\hat p = \P(\R_1 \lrlb \L)$, which by \thref{prop:mc} can be written as $(b/c)\P(\R_1 \lrl \L)$. Substitute this into the formulas for $s$ and $r$ at \eqref{eq:s} and \eqref{eq:r}. Then, substitute these expressions into the first equality at \eqref{eq:prml}. This results in a linear equation in $\P(\R_1 \lrl \L)$. Solving and simplifying this linear equation, which we used mathematical software to carry out, gives the claimed equality. 
\end{proof}

\begin{remark}
If $c=0$, similar formulas as in \thref{lem:partition} could be derived using whichever parameter of $a$ and $b$ is nonzero. For example, if $c=0$ and $a >0$, then we would write the terms involving arrow-arrow collision that partition $\P( (0 \la \b) \wedge \R_1)$ in \eqref{eq:L}
$$\P((0 \la \b) \wedge (\R_1 \lrl \L) ) +  \P((0 \la \b) \wedge (\R_1 \la \L) ) + \P((0 \la \b) \wedge (\R_1 \lrlb \L) )$$
in terms of $\P(\R_1 \la \L)$ using \thref{prop:mc}, rather than in terms of $\P(\R_1 \lrl \L)$ as we did under the assumption $c>0$. The application of \thref{prop:mc} to rewrite \eqref{eq:Rmc} could be handled similarly. 
\end{remark}

% \begin{align}
% w = P(\B_1 \leftarrow \b \mid \B_1)
% \end{align}

%\han{Maybe we can add a note to the reviewers in the resubmission that Lemma 4 handles the main issue in the original subsmission}

In the following derivations, we consider events that occur on systems restricted to various intervals. So moving forward we include all subscripts for clarity.

\begin{lemma} \thlabel{lem:wsr}

Recall that we denote a blockade generated from a $\R_m \lrlb \L_n$ collision by $\h_{m,n}$. Let $\hat p = \P( (\R_1 \lrlb \L)_{(0,\infty)})$. The following identities hold for TCBA: 
\begin{align}
     %w &:= \P(\B_1 \la \b \mid \B_1) + \P(\B_1 \lrl \b \mid \B_1)\\
     %&\qquad =  \f{q (1 - y)}{1 - y q}  \label{eq:w}\\
    s &= \P( (0 \la \b)_{(0,\infty)} \wedge (\R_1 \lrl  \B)_{(0,\infty)}) +  \P( (0 \la \b)_{(0,\infty)} \wedge (\R_1 \lrl  \h)_{(0,\infty)})   \\
    &\qquad \qquad = \f 12  (p+\hat p) (1-\alpha) q^2 \label{eq:s}\\
    r & = \P( (0 \not \la \b)_{(0,\infty)} \wedge (\R_1 \lrl  \B)_{(0,\infty)} ) + \P( (0 \not \la \b)_{(0,\infty)} \wedge (\R_1 \lrl  \h)_{(0,\infty)} )\\
    &\qquad \qquad =  (p+\hat p)(1-\alpha) q (1 - q)  \label{eq:r} 
\end{align}
\end{lemma}

\begin{proof}[Proof of \eqref{eq:s}]
For $m,n\in\Z$ define the indicator random variable  
\begin{align}
    Z(m,n)&=  \ind{{(\R_m\longleftrightarrow\B_n)_{[x_m,\infty)}\wedge(x_n\leftarrow\bullet)_{(x_n,\infty)}}} \\
    & \qquad\qquad  + \sum_{m<k<n} \ind{{(\R_m\longleftrightarrow\h_{k,n})_{[x_m,\infty)}\wedge(x_n\leftarrow\bullet)_{(x_n,\infty)}}}.  
    \end{align}
Note that $Z(m,n)=0$ when $m\ge n$. $Z(m,n)$ describes the ways in which $\R_m$ is annihilated by either $\B_n$ or a $\h$-particle generated by $\L_n$, and subsequently $x_n$ is visited by a left arrow. When the event in the indicator $Z(m,n)$ occurs, we also have $x_m \la \b$. The events in the indicators in the sum $\sum_{n>1} Z(1,n)$ form a partition of the event $\{(0 \la \b)_{(0,\infty)} \wedge ((\R_1 \lrl \B)_{(0,\infty)} \vee (\R_1 \lrl \h)_{(0,\infty)})\}$.
This gives 
\begin{align}
s= \E\sum_{n>1}Z(1,n). \label{eq:right}
\end{align}
Next, let us define ${\vec{\tau}}^{(x_n)}$ and ${\cev{\tau}}^{(x_n)}$ as the times at which $x_n$ is first visited in TCBA restricted to $(-\infty, x_n)$ and $(x_n, \infty)$, respectively. For example, if $0$ is first visited by $\L_k$, then $\cev{\tau}^{(0)} = x_k$. By symmetry ${\vec{\tau}}^{(x_n)}$ and ${\cev{\tau}}^{(x_n)}$ are i.i.d. Moreover by ergodicity, ${\vec{\tau}}^{(x_{n_1})}$ and ${\vec{\tau}}^{(x_{n_2})}$ have the same distribution for $n_1\le n_2$. Since the interdistances between particles are according to a continuous distribution, we have
\begin{align}
\P({\vec{\tau}}^{(x_{n_1})} < {\cev{\tau}}^{(x_{n_2})} \mid {\vec{\tau}}^{(x_{n_1})}, {\cev{\tau}}^{(x_{n_2})} < \infty )= \f 12.\label{eq:1/2}
\end{align}
By \thref{prop:mtp} and the fact that $Z(m,n)=0$ for $m\ge n$, \eqref{eq:right} is equal to
\begin{align}
    \E \sum_{m<1}Z(m,1) 
    &= \sum_{m<1}  \P\left( (\R_m\lrl \B_{1})_{[x_m,\infty)} \wedge (x_{1}\la\L)_{(x_{1},\infty)}  \right) \\
    &\qquad  +\sum_{m<k<1} \P\left( {(\R_m\lrl \hat\b_{k,1} )_{[x_m,\infty)} \wedge (x_{1}\la\L)_{(x_{1},\infty)}} \right) \\ 
    & =  \P\left( (\B_{1})\land (\R \lrlf x_{1})_{(-\infty,x_{1})} \wedge (x_{1}\la\L)_{(x_1,\infty)} \land (\vec{\tau}^{(x_{1})}< \cev{\tau}^{(x_{1})} ) \right) \\
    &\qquad  +  \P\left( \exists \, k<1\,:\, {(\R \lrl \h_{k,1})_{(-\infty,\infty)} \wedge (x_{1}\la \L)_{(x_{1},\infty)}} \right)\\
    % &= \f {p w^2}2 + \P \Big( \exists j <1 \colon (\R \rs \B_{i,j})_{(- \infty, x_j])} \wedge (x_{i,j} \ls \L) \wedge (\vec \tau^{(x_{i,j})} < \cev \tau ^{(x_{i,j})}) \Big) \\
    %&= \f{pw^2}{2} + \f{ \P( \R_1 \perp \L) w^2}{2} 
    &= \f 12  p (1-\alpha)q^{2} + \frac 12  \hat p(1-\alpha)q^2.
    \label{eq:fin}
\end{align}
We will conclude by justifying the equality at \eqref{eq:fin}, which will complete the proof of \eqref{eq:s}.  

For the first term in \eqref{eq:fin}, note that  \eqref{eq:smcf} ensures it is equal to
 \begin{align}
(1-\alpha) \P\left( (\B_{1})\land (\R \ra x_{1})_{(-\infty,x_{1})} \wedge (x_{1}\la \L)_{(x_1,\infty)} \land  (\vec{\tau}^{(x_{1})}< \cev{\tau}^{(x_{1})} ) \right).
 \end{align}
Notice that the events $(\R \lrlf x_1)_{(- \infty, x_1)}$ and $(x_1 \la \L)_{(x_1,\infty)}$ imply that $\vec{\tau}^{(x_1)}, \cev{\tau}^{(x_2)}<\infty$ but provide no further observation regarding the values of these arrival times, that is, they are independent of the event $(\vec{\tau}^{(x_{1})}< \cev{\tau}^{(x_{1})} )$. 
Applying \eqref{eq:1/2} and independence of the one-sided processes gives 
 \begin{align}
\P\left( (\B_{1})\land (\R \lrlf x_{1})_{(-\infty,x_{1})} \wedge (x_{1}\la \L)_{(x_1,\infty)} \land  (\vec{\tau}^{(x_{1})}< \cev{\tau}^{(x_{1})} ) \right) &= \f 12(1-\alpha) p q^2 .
 \end{align}
 
For the second term in \eqref{eq:fin}, it suffices to show that 
\begin{align}
 &\P\left( \exists \, k<1\,:\, {(\R \lrl \h_{k,1})_{(-\infty,\infty)} \wedge (x_{1}\la\L)_{(x_{1},\infty)}} \right) \\
 &\qquad =
     \P\left(  \begin{matrix} \exists \, k<1\,:\,  (\R \lrlf x_{k})_{(-\infty, x_{k})}  \land (\R_{k} \lrlb \L_{1})_{[x_{k},x_{1}]}  \\
\qquad \quad \land (x_{1}\la \L)_{(x_{1},\infty)} 	\land (\vec{\tau}^{(x_{k})}< \cev{\tau}^{(x_{1})} ) \end{matrix} \right) = 
\f 12 (1-\alpha) \hat p q^2\label{eq:fin2}.
\end{align}
 
The first equality in \eqref{eq:fin2} is justified by observing that, for each $k<1$, the event 
\begin{align}
    \left\{ (\R \lrl \h_{k,1} )_{(-\infty,\infty)} \wedge (x_{1}\la\L)_{(x_{1},\infty)} \right\} 
\end{align}
is equal to 
\begin{align}
\left\{    (\R \lrlf  x_{k})_{(-\infty, x_{k})}  \land (\R_{k} \lrlb \L_{1})_{[x_{k},x_{1}]}  
\land (x_{1}\la \L)_{(x_{1},\infty)} 	\land (\vec{\tau}^{(x_{k})}< \cev{\tau}^{(x_{1})} ) \right\}.
\end{align}
Indeed, observe that the right and left arrows have the same constant speed and the blockade $\h_{k,1}$ is at the midpoint $x_{k,1}=(x_{k}+x_{1})/2$ of the interval $[x_{k},x_{1}]$. In order for the blockade created by the event $(\R_{k}\lrlb \L_{1})_{[x_{k}, x_{1}]}$ at location $x_{k,1}$ to be annihilated by a right-arrow, the visit corresponding to $(\R \lrlf x_k)_{(-\infty, x_k)}$ must occur before the first visit to $x_{1}$ by left-arrow.  

For the second equality in \eqref{eq:fin2}, note that the event $(\R_{k} \lrlb \L_{1})_{[x_{k},x_{1}]} $ can occur only for one $k<1$, and the restricted processes on $(-\infty, x_{k})$ and $[x_{k}, x_{1}]$ are independent conditional on $(\R_{k} \lrlb \L_{1})_{[x_{k},x_{1}]}$. It follows that, conditional on $(\R \lrlb \L_{1})_{(-\infty,x_{1}]}$, defining $k$ to be the unique (random) index such that $(\R_{k} \lrlb \L_{1})_{[x_{k},x_{1}]}$ occurs,  $(\R \ra x_{k})_{(-\infty,x_{k})}$ and $(x_{1}\la \L)_{(x_1,\infty)}$ are independent and symmetric in the sense that the random times $\vec{\tau}^{(x_{k})}$ and $\cev{\tau}^{(x_{1})}$ are independent with the same distribution. This yields the second equality in \eqref{eq:fin2}, as desired. 
\end{proof}

\begin{remark} \thlabel{rem:block}
It seems  that the major obstacle to extending TCBA to a four-parameter family that includes the reactions $[\R - \B \implies \B]$ and $[\B - \L \implies \B]$ consists in an asymmetry that arises when comparing certain arrival times of arrows in the one-sided processes. Call a visit to the origin ``strong" if the arrow would destroy a blockade there. Call it ``weak" if a blockade at $0$ would survive the interaction. Let $\vec S$ and $\vec W$ be the times of the first strong and weak visits to $0$ from the left in the one-sided process on $(-\infty,0)$, and similarly for $\cev S$ and $\cev W$. Extending \eqref{eq:s} to these reactions would then require us to compute
\begin{align}\P( (\cev S > \vec S) \wedge (\vec S > \vec W) \mid \cev S, \vec S < \infty  ). \label{eq:y}
\end{align}
Unlike when comparing $\vec \tau$ and $\cev \tau$, there is no obvious symmetry that allows us to compute the value of \eqref{eq:y}. Hence one needs to carry \eqref{eq:y}, as well as an additional term to account for $\R_1$ being destroyed by weakly visiting a blockade, into the equation for $s$ at \eqref{eq:s}. And, eventually solving the main recursion for $q$ becomes intractable without a value for \eqref{eq:y}. A similar difficulty has been observed in the asymmetric BA considered in \cite{JL}. There, $\vec{\tau}$ and $\cev{\tau}$ may have different distribution so $\P(\vec{\tau}\le \vec{\tau})$ cannot be inferred via symmetry. It may depend on the system parameters.
\end{remark}

\begin{proof}[Proof of \eqref{eq:r}]
The proof is similar to \eqref{eq:s}, but uses the modified indicators
\begin{align}
    Z'(m,n)&=  \ind{{(\R_m\longleftrightarrow\B_n)_{[x_m,\infty)}\wedge(x_n\not \la \bullet)_{(x_n,\infty)}}} \\
    & \qquad\qquad  + \sum_{m<k<n} \ind{{(\R_m\longleftrightarrow\h_{k,n})_{[x_m,\infty)}\wedge(x_n\not \la \bullet)_{(x_n,\infty)}}}  
\end{align}
The difference from the indicators used in \eqref{eq:s} is that we require that $x_n$ \emph{not} be visited by a left arrow.

By the definition, $Z'(m,n)=0$ for $m\ge n$. Hence 
	\begin{align}
	   r= \E \sum_{n> 1} Z'(1,n) = \E \sum_{n\in \Z} Z'(1,n) = \E \sum_{m\in \Z} Z'(m,1) = \E \sum_{m<1} Z'(m,1),
	\end{align}
	where the second equality uses the mass transport principle. Using similar reasoning as in the proof of \eqref{eq:s}, this expands to be the claimed formula for $r$. Here is the derivation, but without the similar explanations given in the proof of \eqref{eq:s}:
%
%We give the details  of this expansion in one uninterrupted derivation below:
% \begin{align}
%   \E [\sum_{k>1}Z(1,k)]&=\P({(\R_1\leftarrow\B)_{[x_1,\infty]}\wedge (\B\not\leftarrow\bullet)_{[x_1,\infty)}}) \\&\quad+ \P({(\R_1\longleftrightarrow\B){_{[x_1,\infty]}}\wedge (\B_1\longleftrightarrow\bullet)_{[x_1,\infty)}\wedge(x_1\not\leftarrow\bullet)}_{(x_1,\infty)}) \\&\quad+
%     \P(({\R\longleftrightarrow\B){_{[x_1,\infty]}}\wedge  (x_1\not\leftarrow\bullet)_{[x_1,\infty)}})
% \end{align}
% \newpage
%On the other hand, if we fix $k=1$, we have
\begin{align}
  r%&= \E \sum_{n >1} Z'(1,n)\\
  %&=\E \sum_{m<1}Z'(m,1)\\
  &=\sum_{m<1}\P({(\R_m\lrl \B_1){_{[x_m,\infty)}}\wedge(x_1\not\leftarrow\bullet)}_{(x_1,\infty)} ) \\
  &  \qquad \qquad + \sum_{m<k<1}  \P\left( {(\R_m\lrl \hat\b_{k,1} )_{[x_m,\infty)} \wedge (x_{1}\not \la\L)_{(x_{1},\infty)}} \right) \\ 
& =  \P\left( (\B_{1})\land (\R \lrlf x_{1})_{(-\infty,x_{1})} \wedge (x_{1}\not \la\L)_{(x_1,\infty)}   \right) \\
&\qquad  +  \P\left( \exists \, k<1\,:\, {(\R \lrl \h_{k,1})_{(-\infty,\infty)} \wedge (x_{1}\not \la \L)_{(x_{1},\infty)}} \right)\\
   % &= \f {p w^2}2 + \P \Big( \exists j <1 \colon (\R \rs \B_{i,j})_{(- \infty, x_j])} \wedge (x_{i,j} \ls \L) \wedge (\vec \tau^{(x_{i,j})} < \cev \tau ^{(x_{i,j})}) \Big) \\
    %&= \f{pw^2}{2} + \f{ \P( \R_1 \perp \L) w^2}{2} 
    &=(1-\alpha) p q (1-q) + (1-\alpha)\hat p q (1-q).
\end{align}
\end{proof}

\begin{proposition}\thlabel{prop:q}
For TCBA, it holds for all $p \in [0,1)$ that $0= (1-q)g(p,q)$ with $g$ defined as
\begin{align}
  g(u,v)=  \frac{1-u -2 (1-\alpha) u v - b(1-\alpha) v^2 -(1- (a+b))(1-\alpha) u v^2  }{2-a+b (1-\alpha)(2-v) v }\label{eq:g}.
\end{align}
\end{proposition}

\begin{proof}
    \thref{lem:partition} gives a formula for each term in the partition for $q$ at \eqref{eq:q} that depends only on $p,q,a,b,$ and $\alpha$. This yields $q = G(p,q,a,b,\alpha)$ for an explicit function $G$. Some algebra gives that $0= -q + G(p,q,a,b,\alpha) = (1-q)g(p,q)$ as claimed.
\end{proof}

\section{Properties of $g$} \label{sec:g}

We will derive some elementary, but useful properties of the function $g$ from \eqref{eq:g}. It will be helpful for presenting our arguments to distinguish $p_c$ from its claimed formula in \thref{thm:sc}. So, we introduce the term
\begin{align}
p_* &=  \frac{1 - b (1-\alpha)}{4 -3\alpha - (a+b)(1-\alpha)}.
\label{eq:p*}
\end{align}
Note that the denominator is positive as it can be rearranged as $4-(a+b) - (3-(a+b))\alpha$. With the notation of $p_*$, the first part of \thref{thm:sc} could be restated as $p_c = p_*$ for all TCBA. Note that $p_*$ is derived by solving the equation $g(u,1) =0$. In some sense this corresponds to finding the transition point for $q$ being a root of $1-q$ to being a root of $g(p,q(p))$ seen in Figure~\ref{fig:q}. We record the fact that $p_*$ is the unique such solution in the following lemma.

%\edit{I propose to change all the following lemmas to propositions.}

\begin{lemma} \thlabel{lem:p*}
$u=p_*$ is the unique solution to $g(u,1) =0$. 
\end{lemma}
\begin{proof}
It is easily checked that $g(p_*,1)=0$. Moreover, since $g$ is linear in $u$, it follows that this solution is unique. 
\end{proof}

\begin{lemma} \thlabel{lem:partials}
$\partial_u g(u,v), \partial_v g(u,v) <0$ for all TCBA and $u,v \in [0,1]$.
\end{lemma}

\begin{proof}
 The partial derivatives of $g$ are
 \begin{align}
     \partial_u g(u,v) &= - \frac{v^2 (1-\alpha) (1-(a+b))+2 v (1-\alpha)+1}{2 -a+b (2-v) v (1-\alpha)},\label{eq:u}\\
     \partial_v g(u,v)&=-\frac{2 (1-\alpha) ((2-a) u+b (1-u)) \left((1-a) v+b v^2 (1-\alpha)+1\right)}{(2-a+b (2-v) v (1-\alpha))^2}.\label{eq:v}
 \end{align}
As all parameters lie in $[0,1)$, it is easy to check that both partial derivatives are negative.
\end{proof}

\begin{lemma}\thlabel{lem:g}
\begin{align}
    g(u,v) &>0 \text{ for } (u,v) \in (0,p_*) \times [0,1] \label{eq:g>0} \\
    g(u,v) &<0 \text{ for } (u,v) \in (p_*,1] \times [1,\infty]. \label{eq:g<0}
\end{align}
\end{lemma}

\begin{proof}
    \thref{lem:p*} states that $u=p_*$ is the unique solution to $g(u,1)=0$. Combining this observation with \thref{lem:partials} immediately implies the claimed inequalities. 
\end{proof}

\begin{lemma} \thlabel{lem:roots}
For all TCBA, the equation $g(p, q(p))=0$ has two distinct solutions:
    \begin{align}
    q_{\pm}(p) &:=\frac{-p (1-\alpha) \pm \sqrt{(1-\alpha) \left(b (1-p)^2- p (a (1-p)+p \alpha-1)\right)}}{(1-\alpha) ((1-a) p+b (1-p))}\label{eq:qpm}.
    \end{align}
Moreover, it holds that: 
\begin{enumerate}[label = (\roman*)]
    \item \label{q-<0} $q_-(p) <0$ for $p \in [0,1)$.
   % \item \label{q+>0} $q_+(p) >0$ for $p \in [0,1)$.
    \item \label{q+>1} $q_+(p) >1$ for $p \in [0,p_*)$.
    \item \label{q+<1} $q_+(p) <1$ for $p \in (p_*,1)$.
\end{enumerate}

\end{lemma}

\begin{proof}
Since the numerator of $g$ is quadratic in $v$, the formula for $q_\pm(p)$ follows from the quadratic formula. 
Notice that the discriminant $D$ of \eqref{eq:qpm} can be rewritten as 
\begin{align}
\f{D}{1-\alpha} &= p (1- (a+b))+b (1-p)+ p^2 (a+b)-p^2 \alpha
\end{align}
After multiplying the $p(1-(a+b))$ term by $p$ we have for $p \in [0,1)$ that
\begin{align}
\f{D}{1-\alpha} &> p^2 (1- (a+b)) + b (1-p) + p^2 (a+b) - p^2 \alpha \\
&= b(1-p) + p^2(1-\alpha) \label{eq:D} \\
&>0.
\end{align}
Thus, $q_-$ and $q_+$ are distinct. 

Since the numerator of $q_-$ is negative and the denominator is positive, we immediately observe \ref{q-<0}. 
%To prove Property \ref{q+>0}, notice that the bound at \eqref{eq:D} implies that 
%$$\f{D}{1-x} \geq p^2(1-x).$$ Applying this to $q_+(p)$ we find that the numerator of $q_+(p)$ satisfies 
%$$-p(1-x) + \sqrt D > -p (1-x) + \sqrt{  p^2 (1-x)^2 } = 0,$$
%and so $q_+(p)>0$ for $p \in (p_*,1)$. 
To deduce \ref{q+>1} and \ref{q+<1},  first notice that the uniqueness observation in \thref{lem:p*} along with the fact that $g(p_*,q_+(p_*))=0$ imply that $q_+(p_*) = 1$.  This observation and \thref{lem:partials} together imply \ref{q+>1} and \ref{q+<1}.
\end{proof}

\section{Regularity conditions} \label{sec:regularity}

%Throughout this section, $g$ denotes the function defined in \eqref{eq:g}. The goal of this section is to setup the connection between consistency of $g$ and regularity of the set $I:=\{p \in (p_*, 1): q(p) <1 \}$. 

The goal of this section is to prove the following proposition.

\begin{proposition} \thlabel{prop:clopen}
    $I:=\{ p \in (p_*,1) \colon q(p) <1 \} = (p_*,1)$.
\end{proposition}

\begin{proof}
It follows from \thref{lem:closed} and \thref{lem:open} that $\{ p \in (p_*,1) \colon q(p) <1 \}$ and its complement are both open in the subspace topology on $(p_*,1)$. Thus, $I = \emptyset$ or $I= (p_*,1)$. By \thref{lem:pc+}, $I$ is nonempty so the latter holds.
\end{proof}

\begin{lemma}\thlabel{lem:closed}
$I^c:=\{ p \in (p_*,1) \colon q(p) =1 \}$ is open in the subspace topology on $(p_*,1)$.
\end{lemma}

\begin{proof}

Let $g$  be as in \eqref{eq:g}. \thref{lem:g} gives that there are two distinct solutions $v\in \{q_-, q_+\}$ to $g(p,v) =0$. \thref{lem:g} (i) states that  $q_-<0$ for all $p$. As $q(p)$ is a probability, we cannot have $q(p) = q_-(p)$. The remaining possibilities are that $q=1$ or $q=q_+$. 

Let $q_k = \P( (0 \leftarrow \b)_{[0,x_k]})$. 
The quantities $q_k \uparrow q$ only involve the initial configurations concerning finitely many particles. Conditioning on the velocities of particles in the configuration and integrating to account for interdistances between particles gives a polynomial in $p$. 
%\edit{It definitely took me a while to confirm these two sentences. I think somehow explicitly saying the following is helpful: this probability is $0$ or $1$ conditional on the velocity distribution and the probability of any specific velocity distribution is a polynomial in $p$. Also $q_k \downarrow q$ should be mentioned in a separate clause to avoid confusion as of what involves the initial configuration blabla.} 
It follows from \thref{lem:g} \ref{q+<1} that 
\begin{align}
\text{$q(p) =1$ if and only if $q(p) >q_+(p)$ for $p>p_*$.} \label{eq:iffq}
\end{align}
We use \eqref{eq:iffq} to give the following characterization of the set on which $q(p) =1$:
\begin{align}
    \{ p \in (p_*,1) \colon q(p) =1 \}=  \{ p \in U \colon q(p) > q_+(p) \}  &= \bigcup_{k=1}^\infty \{p \in U \colon  q_k(p) > q_+(p) \}.
    %&=\bigcup_{k=1}^\infty \{p \in (p_*,1) \colon  q_k(p) - \zeta(p) >0 \}. \label{eq:union} 
\end{align}
Continuity of the $q_k$ and $q_+(p)$ ensures that the sets $\{p \in U \colon q_k(p) > q_+(p)\}$ are open. Thus, so is the union which is equal to $I^c$. 
\end{proof}

\subsection{Superadditivity}
The idea underlying the upcoming \thref{lem:open}, that $I$ is open, was developed in \cite{ST}, and extended in \cite[Lemma 10]{HST}. A more general version tailored to the asymmetric setting is proven in \cite{JL}. Coalescence makes for new challenges. Before getting to the lemma, we introduce some additional notation.

Let $$B(j,k) = \sum_{i=j}^k \ind{ (\B_i \text{ survives})_{[x_j,x_k]} }$$ be the number of blockades that survive in the process restricted to $[x_j,x_k]$. Note that $B(j,k)$ only counts surviving blockades from the initial configuration of particles; blockades generated from $\R - \L \implies \B$ reactions do not contribute. We use the quiver interpretation of TCBA described in Section \ref{sec:quiver}. 

To briefly summarize, each arrow in the process, call such particles \emph{original arrows}, carries a \emph{quiver} of $\text{Geometric}(1-\alpha)$ many \emph{sharp arrows}. When the arrow meets a blockade one of the sharp arrows mutually annihilates with the blockade, but the original arrow along with its quiver of any remaining sharp arrows continue moving. The original arrow is destroyed at the moment its last sharp arrow is destroyed. When two original arrows meet, some arbitrary rule is used to decide the reaction. If an original arrow is destroyed in a $\R - \L$ collision, its quiver of sharp arrows is also destroyed. 

In the quiver formulation, we view sharp arrows as distinct particles that follow the same trajectory as the original arrow. These arrows track how many blockades the original arrow could potentially destroy. Let $A(j,k)$ be the number of left and right sharp arrows that survive in $[x_j, x_k]$.  Define the count
\begin{align}
N(j,k) &= B(j,k) - A(j,k).  \label{eq:quiver}
\end{align}
This is in some sense a worst-case weighting of surviving blockades; every surviving sharp arrow is treated like it will ultimately destroy a blockade.

We now show that $N$ is superadditive after merging intervals.
%Our lemma is more robust than what is used in \cite{HST}.
In the analogue \cite[Lemma 12]{HST}, it is required that the right interval have no surviving right arrows. We show that, with our more general weighting scheme, superaddititivy holds with no hypotheses about surviving arrows. This extra level of generality ends up being crucial for proving the upcoming \thref{lem:iff}. See \thref{rem:sa}.

\begin{lemma} \thlabel{lem:sa}
Let $k < \ell$ be positive integers. For any initial assignment of particle types and spacings to $(\b_i)_{i \in \mathbb Z}$ we have
\begin{align}N (1, \ell) \geq  N (1, k) + N (k + 1, \ell).\label{eq:sa}
\end{align}
\end{lemma}

\begin{proof}

Let $I = [x_1,x_k]$ and $J= [x_{k+1}, x_\ell]$. %The idea of the proof is to track the effects of surviving sharp arrows from $I$ and $J$ in the combined process on $I \cup J$.
We show that when sharp arrows that survived in the process on $I$ or $J$ are destroyed, this has at worst a net-neutral effect on the difference between surviving blockades and sharp arrows. We track these changes in real time.

%Index blockades and sharp arrows in some manner so that each has a unique identifier.
%in the quiver of $\b_i$ by the instruction $\dot X_i(m)$ they correspond to. 
Using any index system that uniquely identifies blockades and sharp arrows, define $\BB^I, \BB^J, \AA^I$, and $\AA^J$ to be the sets of blockades and sharp arrows that survive in the processes restricted to $I$ and $J$. The sets $\BB^I$ and $\BB^J$ only include blockades from the original configuration; blockades generated from $\R - \L \implies \B$ reactions are not counted.
%Let $\BB_t$ be the set of blockades from $\BB^I \cup \BB^J$ that are still surviving at time $t$ in the combined process on $I \cup J$.
We will define a pair of set-valued processes $(\AA_t,\BB_t)$ with the following properties:
%\edit{Would it be better to define $(\AA_t, \BB_t)$ as a pair of processes? -Lily}
%We will define a set-valued process $\AA_t$ such that:
    \begin{enumerate}[label = (\roman*)]
        \item $\BB_t$ is the set of blockades from $\BB^I \cup \BB^J$ that are still surviving at time $t$ in the combined process on $I \cup J$. %This does not count any blockades generated by $\R - \L \implies \B$ collisions.
        \item $\AA_0 = \AA^I  \cup \AA^J$.
         \item $\mathcal A_t$ is non-increasing and $\BB_t$ may decrease only when a decrease of the same magnitude occurs in $\AA_t$.
        \item For $T:= x_\ell - x_1$, we have $|\BB_T| = B(1,\ell)$ and  $|\AA_T| = A(1,\ell)$.
    \end{enumerate}
Using the characterization of $N(1,\ell)$ at \eqref{eq:quiver}, these properties give \eqref{eq:sa} since
\begin{align}
    N(1,k) + N(k+1,\ell) \overset{(i), (ii)}= |\BB_0| - |\AA_0| 
    \overset{(iii)}\leq |\BB_T| - |\AA_T| \overset{(iv)}= N(1,\ell).
\end{align}
It remains to define $(\AA_t,\BB_t)$ and prove that it satisfies (i), (ii), (iii), and (iv). 

We take property (i) as the definition of $\BB_t$ and (ii) as the definition of $\AA_0$. Let $t_0=0$ and $t_1>t_0$ be the first time in the combined process on $I \cup J$ that a sharp arrow, say $\p$, from $\mathcal A_{t_0}$ is annihilated. If the collision involves two sharp arrows from $\mathcal A_{t_0}$, let $\p$ be the left arrow. There are three possible ways that $\p$ is destroyed:
\begin{enumerate}[label = (\Roman*)]
    \item If $\p$ is annihilated by a blockade $\B_i
    $ counted by $\BB_{t_0}$, then set $\AA_{t_1} = \AA_{t_0} \setminus \{\p\}$ and $\BB_{t_1} = \BB_{t_0} \setminus \{\B_i\}$. %In this case both $\BB_{t_1}$ and $\AA_{t_1}$ have decreased by 1. 
    \item If $\p$ is annihilated from hitting a blockade $\B$ that does not belong to $\mathcal B_{t_0}$, then $\BB_{t_1}=\BB_{t_0}$.
        \begin{enumerate}[label = (\alph*)]
            \item Let $\p'$ be the first sharp arrow that: reaches the location of $\B$ in the process on $I \cup J$, has the opposite direction of $\p$, and is not counted by $\AA_{t_0}$. Set $\AA_{t_1} = (\AA_{t_0} \setminus \{\p\} ) \cup \{ \p'\}$.
            \item If there is no such $\p'$,
            %sharp arrow is unleashed (for example, $\B$ was generated from a $\R - \L \implies \B$ collision and is not counted by $\BB_{t_0}$), 
            then set $\AA_{t_1} = \AA_{t_0} \setminus \{\p\}.$ %Note that there can only be one such arrow since only one sharp arrow can destroy $\B$.
        \end{enumerate}
    In this case, the blockade $\B$ is generated by an arrow-arrow collision between times $t_0$ and $t_1$. In (II)(a), we account for the case that $\p'$ would have been annihilated by $\B$ in the combined process if not for $\p$. Note that $\p'$ may or may not have already been in $\AA_{t_0}$.
    \item If $\p$ is destroyed by another arrow $\p''$, then define $\mathcal A_{t_1}$ to be $\AA_{t_0}$ minus all of the sharp arrows in the same quiver as $\p$ as well as---if the reaction is mutual annihilation and $\p'' \in \mathcal A_{t_0}$---all arrows in the same quiver as $\p''$. Set $\BB_{t_1} = \BB_{t_0}$.
\end{enumerate}

Iterate this procedure by considering the next sharp arrow from $\AA_{t_j}$ to be destroyed at time $t_{j+1}>t_j$. This gives new values of $\AA_{t_{j+1}}$ and $\BB_{t_{j+1}}$ according to whichever of (I), (II), or (III) occurs. We make the process right continuous for $t \geq 0$ by setting $\AA_t = \AA_{t_j}$ for all $t\in[ t_{j},t_{j+1})$ and $j \geq 0$. 
By construction, we have (iii). 

To show (iv), first notice that, since arrows have unit speed, $\AA_t$ and $\BB_t$ must fixate after the length of the interval $T= x_{\ell} - x_1$ has elapsed. $\AA_t$ tracks all potentially surviving sharp arrows. By construction, $\AA_T$ does not contain non-surviving sharp arrows; when a sharp arrow is annihilated, it is removed from the process. The only time a new sharp arrow is added to $\AA_t$ is when (II.a) occurs, in which case the added sharp arrow has survived thus far. 

Next, we show that \emph{all} surviving sharp arrows in the combined process are contained in $\AA_T$. First observe that $\AA_T$ contains all sharp arrows from $\AA_0$ that survive in the combined process, as an arrow is removed only if it is annihilated. If a sharp arrow survives in the combined process but is not in $\AA_0$, then it must be because the particle that it was going to be annihilated by in the separate process is annihilated by an arrow from the other interval in the combined process. Recall that although blockades can be generated through collisions, arrows cannot. In fact, the ``unleashing'' of an arrow is only possible when it was going to be annihilated with a blockade and this is captured in case (II.a) when the unleashed arrow is added to $\AA_t$. This gives the first part of (iv), that $\AA_T = A(1,\ell)$.

To show the second half of (iv), that $|\BB_T|=B(1,l)$, we first note that $\BB_0$ contains all possible surviving blockades in the combined process -- non-surviving blockades in the separate processes would be destroyed by the same sharp arrows that destroyed them in the separate processes if not by other arrows from opposite intervals prior. Since the dynamic construction of $\BB_t$ only removes annihilated blockades, $\BB_T$ consists of all surviving blockades by the end of the combined process.
\end{proof}

%By construction we have (iii). Towards (iv), first notice that, since arrows have unit speed, $\AA_t$ and $\BB_t$ must fixate after the length of the interval $T= x_{\ell} - x_1$ has elapsed. Moreover, we have $|\AA_T| = A(1,\ell)$ since $\AA_t$ tracks what occurs with each surviving sharp arrow from $\AA^I$ and $\AA^J$. Each such arrow either still survives, or is destroyed in some manner in the combined process. When (II.a) occurs, the destruction of a sharp arrow allows for exactly one other sharp arrow to survive longer in the combined process. This is accounted for with $\AA_t$, and so $|\AA_T| = A(1,\ell)$. The only arrows that can destroy blockades in $\BB_0$ are surviving sharp arrows from $\AA_0$, or sharp arrows that survive longer in the combined process due to the annihilation of some other sharp arrow from $\AA_t$ as in (II.a). The quantity $\BB_t$ tracks these changes, and thus $|\BB_T| = B(1,\ell)$. This gives (iv). 

\subsection{A necessary and sufficient condition for blockade survival}
\begin{lemma} \thlabel{lem:iff}
Let $N_k = N(1,k)$ with $N$ defined at \eqref{eq:quiver}. For all $p \in (0,1)$ it holds that
\begin{align}
\theta(p) >0 \iff \text{ there exists $k\geq1$ with } \mathbb E N_k >0 \label{eq:iff}. 
\end{align}
\end{lemma}

\begin{proof}
The forward implication is analogous to \cite[Proposition 11]{HST}. The key observation is that for any $x_i$ in an interval $I$, $P ((\B_i \text{ survives})_I)$ is decreasing in $I$. This continues to hold with the coalescence rules in TCBA. To see why, suppose that we have a configuration of particles in $I$ with $(\B_i \text{ does not survive})_I$. Keeping this configuration fixed, there is no manner in which one could add particles outside of $I$ to intercept the particle that destroys $\B_i$ in time.
%Hence if we write   
% \begin{align}
%     B_{k} &= \sum_{i=1}^{k} \mathbf{1}\{\text{$\B_{i}$ survives)$_{[x_{1},x_{k}]}$}\},
%   %  \overline{B}_{k} &= \sum_{i=1}^{k} \mathbf{1}(\text{$\bullet_{i}=B_{i}$ and $\B_{i}$ survives from $\mathbb{R}$}),
% \end{align}
This monotonicity ensures that %$B_{k}\ge \overline{B_{k}}$. It follows that 
\begin{align}
\E B(1,k) = \sum_{i=1}^k \P((\B_i \text{ survives})_{[x_1,x_k]})\geq \sum_{i=1}^k \P((\B_i \text{ survives})_{\mathbb R})= %\mathbb E \overline{B}_k = 
%\sum_{i=1}^{k} p\, \mathbb{P}(\text{$\B_{i}$ survives from $\mathbb{R}$}\,|\, \bullet_{i}=\B_{i}) = 
kp\,\theta(p) \to \infty 
\end{align}
with $B(1,k)$ from \eqref{eq:quiver}.
%Here $\E_{\mathbb R} B_k$ is the expected number of blockades to survive from $[x_1,x_k]$ in TCBA with all particles in $\mathbb R$ included. 
%\edit{I think the equality could use some explanation.} \han{Explanation added} \lily{Oh the initial $\E B_k \geq \E_{\mathbb R} B_k$ was clear, but $kp\theta(p)$ may take people a beat to understand.}

Let $\cev A_k$ be the sharp left arrows counted by $A(1,k)$ from \eqref{eq:quiver} and $\vec A_k$ the sharp right arrows  so that $A(1,k) = \cev A_k + \vec A_k$. We will next show that $\E \cev A_k$ and $\E \vec A_k$ are bounded. Symmetry ensures that $\E \cev A_k = \E \vec A_k$, so we only provide the argument for $\cev A_k$. For each $i\ge 1$, the event that $\L_{i}$ survives from the restriction to $[x_{1},x_{k}]$ is non-decreasing in $k$, since such arrows do not interact with any particles to their right. This monotonicity implies that  $\cev A_k \uparrow \cev A_\infty$.
%\lily{Ok this makes sense. But this argument doesn't apply to $\vec A_k$ so is the symmetry still true?} 
Let $T$ denote the number of surviving left arrows counted from $\cev{A}_{\infty}$ so that $\L_{i_1},\hdots, \L_{i_T}$ reach $0$ for some integers $i_{1},\ldots,i_{T}\ge 1$. When $\theta(p)>0$, by using a renewal property of TCBA, $T$ is a geometric random variable supported on $0,1,\hdots$ with parameter $1-q$. Hence we have 
\begin{align}
   \cev A_k \uparrow  \cev A_\infty:= \textstyle \sum_{j=1}^T G_j ,
\end{align}
where each $G_j$ an independent geometric random variable supported on $1,2,\hdots$ with parameter $1-\alpha$. The $G_j$ count how many additional sharp arrows remain in the quiver of the $j$th arrow to arrive to $0$. $G_j$ has a geometric distribution because of the memoryless property. %\lily{How do we know $\cev A_k \uparrow \cev A_\infty$ is true?} with $T$ a geometric random variable with parameter $1-q$, each $G_j$ an independent geometric random variable with parameter $1-x$, and $Y_j$ a Bernoulli random variable with success probability $y$. Here $T$ represents the left arrows $\L_{i_1},\hdots, \L_{i_T}$ that reach $0$. Conditional on $0 \leftarrow \L_{i_j}$, this is the claimed geometric random variable by the memoryless property.  
Hence, 
\begin{align}
\mathbb E A(1,k) = 2 \mathbb E \cev A_k \leq \f{ 2}{1-q} \f 1 {1-\alpha} < \infty.
\end{align}
Since $\E B(1,k)\rightarrow \infty$ as $k\rightarrow \infty$, we have $\E N_{k}=\E B(1,k) - \E A(1,k)>0$ for some large $k$. This shows the forward implication of \eqref{eq:iff}.

%{\color{red} Omit: It follows that for large enough $k$, we have $\mathbb E A(1,k) >0$.} \lily{What is $A_k$ and why $\mathbb E \cev A_k < \infty$ implies $\mathbb E A_k>0$?} \han{I don't think $\mathbb{E} A(1,k)>0$ is not what we need, we need $\mathbb{E} A(1,k)<\infty$.}

Towards proving the reverse implication of \eqref{eq:iff}, suppose that $k\geq 1$ is such that $\mathbb E N_k >0$. We reveal the configuration on $k$ consecutive particles at a time. To this end, let $\tilde N^{(i)} = N((i-1)k+1,ik)$ and denote  $\tilde{S}_{n}=\sum_{i=1}^{n} \tilde N^{(i)}$, $n\ge 1$.   Notice that the $\tilde N^{(i)}$ are i.i.d.\ with common distribution $N_k$ and $\E N_k = \E \tilde{N}^{(1)}>0$. Hence $(\tilde{S}_{n}-\tilde{S}_{1})_{\ge 1}$ is a random walk with positive drift, so it stays strictly above $0$ with  positive probability. Independently from this event, with positive probability the first $k$ particles right of $0$ are blockades, i.e.,  $\tilde N^{(1)} = k$. Applying \thref{lem:sa} gives that
\begin{align}N(1,nk) \ge N^{(1)} + (\widetilde{S}_{n}-\widetilde{S}_{1}) >k,\quad \text{ for all } n \geq 2\label{eq:saa}
\end{align}
with positive probability. 

On the event that \eqref{eq:saa} holds, we claim that $0$ is never visited by a left-moving particle. Indeed, consider the $(n+1)$st reveal of the $k$ particles from the interval $J_{n+1} := [x_{1+nk}, x_{(n+1)k}]$. Since $N(1,nk) >k$, there are least $k+1$ surviving blockades in $I_n:=[x_1, x_{nk}]$. If $0$ is reached by one of the particles from $J_{n+1}$, then no blockades survive in $I_n$. However, at most $k$ new blockades are introduced by $J_{n+1}$. So, if this occurs, we would have $N(1,(n+1)k) \leq k$, which contradicts \eqref{eq:saa}. It follows that $\theta(p)>0$ since $0$ is never visited with at least the probability that \eqref{eq:saa} occurs.
\end{proof}

\begin{remark}\thlabel{rem:sa}
When we apply \thref{lem:sa} at \eqref{eq:saa} it is important that we have a formulation of superadditivity that accounts for the effects of surviving right arrows. Unlike the approach used in \cite{HST}, we cannot extend each interval of size $k$ a random amount to eliminate any surviving right-arrows without changing the distribution of $\tilde N^{(i)}$. Indeed, expanding the interval until the surviving right arrows are removed may bring in left arrows that not only destroy the surviving right arrows, but also some of the blockades in the initial segment. Our formulation of subadditivity could be applied to the argument in \cite{HST} and would slightly simplify their proof of the analogue of \thref{lem:iff}.  

One might further wonder why we can get away with what appears to be a stronger necessary condition ($\theta>0 \implies \exists k \colon \E[B_k - (\vec A_k + \cev A_k)] > 0$) than that observed in \cite{HST} ($\theta>0 \implies \exists k \colon \E [B_k - \cev A_k] >0$). %The reason is that, when $\theta>0$, the surviving arrows in a given interval are dust particles among a fortress of surviving blockades. 
As the proof of \thref{lem:iff} reveals, there are $O(1)$ blockades and $\Omega(n)$ surviving blockades. Thus, it does not matter if we include the worst-case remaining impact from the ``dust'' of surviving arrows. For similar reasons we do not need to include the blockades resulting from $\R - \L \implies \B$ reactions in $B_k$. 
\end{remark}

\begin{lemma} \thlabel{lem:open}
$\{ p \in (0,1) \colon q(p) <1 \}$ is open in the subset topology on $(0,1)$.
\end{lemma}

\begin{proof}

Using \thref{lem:iff}, we have 
$$\{ p \in (p_*,1) \colon \theta(p) >0\} = \bigcup_{ k \geq 1} \{ p \in (p_*,1) \colon \E N_k >0\}.$$
The function $p \mapsto \E N_k$ is easily seen to be a finite polynomial in $p$, and so the sets in the union are open. 
\end{proof}

\begin{lemma} \thlabel{lem:pc+}
For TCBA, it holds that $$p_c^+\leq \frac{1}{2-\alpha} <1.$$
\end{lemma}

\begin{proof}

By \thref{lem:iff}, it suffices to prove that $\E N_1 >0$ for $p$ large enough. We can easily compute
$$\E N_1 = p - (1-p) \f{1}{1-\alpha}.$$
This is strictly greater than $0$ for $p > 1/(2- \alpha)<1$ as claimed. 
\end{proof}

\section{Proof of \thref{thm:sc}} \label{sec:main} 
\begin{proof}
Let $p_*$ be as defined at \eqref{eq:p*}. We will write $q$ in place of $q(p)$ and similarly for the functions $q_-$ and $q_+$ from \thref{prop:q}. Recall that \thref{prop:q} gives that $v=q$ is a solution to the equation $(1-v)g(p,v) = 0$ which has solutions $\{1, q_-, q_+\}$. \thref{lem:g} \ref{q-<0} states that $q_- <0$. Consequently, $q$ must equal either $1$ or $q_+$. For $p \in [0,p_*)$, \thref{lem:g} \ref{q+>1} states that $q_+ >1$. We immediately deduce that $q =1$ for $p \in [0, p_*)$. Next, \thref{prop:clopen} states that $q<1$ for $p \in (p_*,1)$. As $q_- <0$, this leaves $q = q_+$ as the only possible value for $q$ for $p$ in the interval $(p_*,1)$. Lastly, \thref{lem:p*} implies that $q_+(p_*) = 1$, and thus $q(p_*)$, which must equal either $1$ or $q_+(p_*)$, also equals $1$. This gives the claimed value of $p_c$ and claimed formula for $q(p)$. 

\end{proof}

                                                               %%
%%%%%%%%%%%%%%%%%%%%%%%%%%%%%%%%%%%%%%%%%%%%%%%%%%%%%%%%%%%%%%%%%%%

\bibliographystyle{amsplain}
%\bibliography{BA}

% add below the content of your .bbl file produced by bibtex.

\providecommand{\bysame}{\leavevmode\hbox to3em{\hrulefill}\thinspace}
\providecommand{\MR}{\relax\ifhmode\unskip\space\fi MR }
% \MRhref is called by the amsart/book/proc definition of \MR.
\providecommand{\MRhref}[2]{%
  \href{http://www.ams.org/mathscinet-getitem?mr=#1}{#2}
}
\providecommand{\href}[2]{#2}

%%%%%%%%%%%%%%%%%%%%%%%%%%%%%%%%%%%%%%%%%%%%%%%%%%%%%%%%%%%%%%%%%%%
%%                                                               %%
%% You may add acknowledgments (optional).                       %%
%%                                                               %%
%%%%%%%%%%%%%%%%%%%%%%%%%%%%%%%%%%%%%%%%%%%%%%%%%%%%%%%%%%%%%%%%%%%
\begin{acks}
This research project was partially completed during the 2020 Baruch Discrete Mathematics REU, supported by NSF awards DMS-1802059, DMS-1851420, and DMS-1953141.
\end{acks}

%%%%%%%%%%%%%%%%%%%%%%%%%%%%%%%%%%%%%%%%%%%%%%%%%%%%%%%%%%%%%%%%%%%
%%                                                               %%
%% You have reached the end of your document.                    %%
%%                                                               %%
%%%%%%%%%%%%%%%%%%%%%%%%%%%%%%%%%%%%%%%%%%%%%%%%%%%%%%%%%%%%%%%%%%%

\end{document}